\documentclass[12pt]{article}
\usepackage{fullpage,amssymb,amsfonts,amsmath,amstext,amsthm,amscd,graphicx,verbatim}

\begin{document}

\newtheorem{theorem}{Theorem}[section]
\newtheorem{definition}[theorem]{Definition}
\newtheorem{lemma}[theorem]{Lemma}
\newtheorem{proposition}[theorem]{Proposition}
\newtheorem{corollary}[theorem]{Corollary}
\newtheorem*{thmA}{Theorem A}

\theoremstyle{definition}
\newtheorem*{remarks}{Remarks}

\numberwithin{equation}{section}

\newcommand{\x}{{\bf x}}
\newcommand{\y}{{\bf y}}
\newcommand{\0}{{\bf 0}}
\newcommand{\p}{{\bf p}}
\newcommand{\q}{{\bf q}}
\newcommand{\Z}{\mathbb{Z}}
\newcommand{\R}{\mathbb{R}}
\newcommand{\C}{\mathbb{C}}
\newcommand{\N}{\mathbb{N}}

\newcommand{\Real}{\mathop{\rm Re}\nolimits}
\newcommand{\Imag}{\mathop{\rm Im}\nolimits}
\newcommand{\diam}{\mathop{\rm diam}\nolimits}

\newcommand{\eqn}{\begin{equation}}
\newcommand{\eqnend}{\end{equation}}

\title{Iteration of quasiregular tangent functions in three dimensions}
\author{A. N. Fletcher and D. A. Nicks}

\maketitle

\begin{abstract}
We define a new quasiregular mapping $T:\R^3\to\R^3\cup\{\infty\}$ that generalizes the tangent function on the complex plane and shares a number of its geometric properties. We investigate the dynamics of the family $\{\lambda T:\lambda>0\}$, establishing results analogous to those of Devaney and Keen for the meromorphic family $\{z\mapsto\lambda\tan z:\lambda>0\}$, although the methods used are necessarily original.\\

\end{abstract}

\section{Introduction}

In the study of iteration of meromorphic functions on the complex plane, the tangent function is one of very few examples where the Julia set has a simple form. Devaney and Keen \cite{DK} proved the following result about the dynamics of the tangent family $\tau_\lambda(z)=\lambda\tan z$.

\begin{thmA}[\cite{DK}]
If $\lambda>1$, then the Julia set $J(\tau_\lambda)$ is the real line. On the upper and lower half-planes, the iterates of $\tau_\lambda$ converge to $i\xi_0$ and $-i\xi_0$ respectively, where $\xi_0$ is the unique positive solution to
\eqn \xi_0=\lambda\tanh\xi_0. \label{defn xi0} \eqnend

If $\lambda=1$, then $J(\tau_\lambda)=\R$ and the forward orbit of any point with non-zero imaginary part converges to the parabolic fixed point at the origin.

If $0<\lambda<1$, then $J(\tau_\lambda)\subseteq\R$ is locally a Cantor set, and the Fatou set is the infinitely connected basin of attraction of the fixed point at the origin.
\end{thmA}

Devaney and Keen also considered complex values of $\lambda$ with $0<|\lambda|<1$. A detailed classification of the dynamics of the family $\{\lambda\tan z: \lambda\in\C,\,\lambda\ne0 \}$ has since been given by Keen and Kotus \cite{KK}.

In this article, we will consider a family of mappings on $\R^3$ that is a natural generalization of the meromorphic tangent family. This is motivated both by the results of Devaney and Keen mentioned above, and also by the higher dimensional analogues of the sine function that are constructed in \cite{BE}, \cite[\S2]{Dr} and \cite[p.111]{M2}. Bergweiler and Eremenko \cite{BE} demonstrated a seemingly paradoxical decomposition of $\R^n$ via iteration of these trigonometric analogues.

The mappings of $\R^n$ studied in \cite{BE, Dr, M2} and this paper are all quasiregular functions. We recall that a continuous function $f:U\to\R^n$ on a domain $U\subseteq\R^n$ is called \emph{quasiregular} if it belongs to the Sobolev space $W^1_{n,\mathrm{loc}}(U)$ and if there exists $K\ge1$ such that
\[ \|Df(\x)\|^n \le KJ_f(\x) \quad \mbox{for almost every } \x\in U, \]
where $\|Df(\x)\|$ is the norm of the derivative of $f$ and $J_f(\x)$ denotes the Jacobian determinant. More generally, a continuous function $f:\R^n\to\R^n\cup\{\infty\}$ is called quasiregular, or sometimes \emph{quasimeromorphic}, if the set of poles $f^{-1}(\infty)$ is discrete and if $f$ is quasiregular on $\R^n\setminus f^{-1}(\infty)$. Informally, a quasiregular function is one that maps infinitesimal spheres to infinitesimal ellipsoids with bounded eccentricity. Quasiregular maps are a generalization of analytic and meromorphic functions on the plane; see Rickman's monograph \cite{R} for many more details.

The rich theory that surrounds the dynamics of entire and meromorphic functions has prompted an investigation of the iterative behaviour of quasiregular mappings, see for example \cite{BergFJ, FN2, HMM, M1, S} and also the survey article \cite{BergCMFT}. One purpose of this article is to help address the lack of examples of quasiregular maps with well-understood dynamics. 

Identifying useful definitions for the Fatou and Julia sets of a general quasiregular function can be a complicated matter, see \cite{BergCMFT, BergFJ,  SY3} and \cite[Chapter~21]{IM} for results in this direction. The escaping set of a quasiregular map provides a more easily defined dynamically interesting set, and has been studied in \cite{BFLM, FN, N}. We recall that, for a function $f:\R^n\to\R^n\cup\{\infty\}$, the $k$th iterate $f^k$ is not defined at the poles of $f^{k-1}$. Denoting the backward orbit of infinity by
\[ O^-(\infty) = \bigcup_{k=1}^\infty f^{-k}(\infty), \]
we see that the family of iterates $\{f^k:k\in\N\}$ is only defined on $\R^n\setminus O^-(\infty)$. We thus define the \emph{escaping set} as
\[ I(f) = \{\x\in\R^n\setminus O^-(\infty) : f^k(\x)\to\infty \mbox{ as } k\to\infty\}; \]
that is, the set of points whose iterated images tend to, but never land at, infinity. This set is known to be non-empty for a large class of quasiregular mappings \cite{BFLM, FN}; in particular, $I(f)\ne\emptyset$ when $f$ is quasiregular and has infinitely many poles.

The escaping set is playing an increasingly important role in complex dynamics. Results of Eremenko \cite{E} and Dom\'{\i}nguez \cite {Do} together show that the boundary $\partial I(f)$ coincides with the Julia set $J(f)$ for any meromorphic function $f$. Further results on the escaping sets of entire and meromorphic functions can be found in \cite{BRS, Rempe, RS1, RS2, RRRS} and elsewhere. We will demonstrate that the boundary of the escaping set of the quasiregular tangent function that we construct has many of the properties typically expected of a Julia set. 


\section{Statement of results}\label{sect:results}

In Section \ref{sect:T}, we construct a quasiregular mapping $T:\R^3\to\R^3\cup\{\infty\}$ that is in many ways a generalization of the meromorphic tangent function on the complex plane. This mapping $T$ is doubly-periodic, it has a pair of omitted asymptotic values and it shares several of the geometric properties of the tangent function. Moreover, the restriction of $T$ to either the $(x,z)$-plane or the $(y,z)$-plane yields precisely the standard tangent function. See Section~\ref{sect:T} for details. We remark that the construction is valid in $\R^n$, for $n \geq 3$, but for simplicity we shall restrict to the case $n=3$.

For $\lambda>0$, we put 
\[ T_\lambda(\x)=\lambda T(\x). \]
We study the dynamics of this one-parameter family with the aim of establishing an analogue of Theorem A. We use the following definitions for continuous functions $f:\R^n\to\R^n\cup\{\infty\}$. We call ${\bf x}_0 \in\R^n$ a \emph{fixed point} of $f$ if $f({\bf x}_0)={\bf x}_0$ and define its \emph{basin of attraction} to be 
\[ \mathcal{A}({\bf x}_0) = \{\x\in\R^n : f^k(\x)\to {\bf x}_0 \mbox{ as } k\to\infty\}. \]
Furthermore, ${\bf x}_0$ is said to be an \emph{attracting} fixed point if there exists $c<1$ such that $\|f({\bf y})-{\bf x}_0\|<c\|{\bf y}-{\bf x}_0\|$ for all ${\bf y}$ in some punctured neighbourhood of ${\bf x}_0$.

Our first result shows that the iterative behaviour of $T_\lambda$ in the upper and lower half-spaces of $\R^3$ is analogous to that of $\lambda\tan z$ on the upper and lower half-planes as described in Theorem A.

\begin{theorem}\label{thm:attr}
The function $T$ has a fixed point at the origin. If $\lambda>1$, then $T_\lambda$ has attracting fixed points at $(0,0,\pm\xi_0)$ with basins of attraction
\[ \mathcal{A}((0,0,\pm\xi_0)) = \{(x,y,z):\pm z>0\}, \]
where $\xi_0>0$ is as in \eqref{defn xi0}.

If $0<\lambda<1$, then $T_\lambda$ has an attracting fixed point at the origin. Moreover, when $0<\lambda\le1$ we have $\{(x,y,z):z\ne0\}\subseteq\mathcal{A}(\0)$.
\end{theorem}

Our attention turns next to the dynamics of $T_\lambda$ on the invariant $(x,y)$-plane. We recall that, for a meromorphic function with poles, the Julia set is equal to both the boundary of the escaping set and also the closure of the backward orbit of infinity. In particular, for $\tau_\lambda(z)=\lambda\tan z$ with $\lambda>1$, Theorem A shows that $\partial I(\tau_\lambda) = \overline{I(\tau_\lambda)}=\R$. We remark that by Theorem \ref{thm:attr}, 
\[ I(T_\lambda)\subseteq\{(x,y,0):x,y\in\R\} \quad \mbox{for all } \lambda>0,\] and so $\partial I(T_\lambda) = \overline{I(T_\lambda)}$.

\begin{theorem}\label{thm:O-(infty)}
If $\lambda>0$, then $\overline{I(T_\lambda)}=\overline{O^{-}(\infty)}$. The escaping set $I(T_\lambda)$ is uncountable, totally disconnected and has no isolated points.
\end{theorem}

\begin{theorem}\label{thm:sqrt2}
If $\lambda>\sqrt{2}$, then $\overline{I(T_\lambda)}=\{(x,y,0):x,y\in\R\}$. The constant $\sqrt{2}$ here cannot be replaced by any smaller value.
\end{theorem}

We also show how the value of $\lambda$ determines the connectedness of $\overline{ I(T_{\lambda})}$.

\begin{theorem}\label{thm:connlocus}
If $\lambda \geq 1$, then $\overline{I(T_{\lambda})}$ is connected. 
If $\lambda <1$, then $\overline{I(T_{\lambda})}$ is not connected.
\end{theorem}

The following related questions remain open: in analogy to Theorem A, are there values of $\lambda$ for which $\overline{I(T_\lambda)}$ is locally a Cantor set, or for which $\overline{I(T_\lambda)}$ and $\mathcal{A}(\0)$ form a partition of~$\R^3$?

Our final results reflect two well-known properties of the Julia set of a meromorphic function $f$: firstly, periodic points are dense in the Julia set; and secondly, that if an open set $U$ meets $J(f)$ then $\bigcup_{k\in\N}f^k(U)$ omits at most two points of $\C\cup\{\infty\}$. The proposition that a version of this `blowing-up property' could be used to define Julia sets for certain wide classes of quasiregular functions is explored in \cite{BergCMFT, BergFJ, SY3}. We observe here that $\overline{I(T_\lambda)}$ possesses a strong form of this property.

\begin{theorem}\label{thm:expanding}
Let $\lambda>0$ and let $U$ be an open set that intersects $\overline{I(T_\lambda)}$. Then, for some $m>0$, 
\[ (T_\lambda)^m\left(U\setminus E_m\right) = \left(\R^3\cup\{\infty\}\right)\setminus\{(0,0,\pm\lambda)\}, \]
where $E_m=\bigcup_{n=1}^{m-1}(T_\lambda)^{-n}(\infty)$.
\end{theorem}

\begin{theorem}\label{thm:periodic}
For all $\lambda>0$, the closure of the set of periodic points of $T_\lambda$ contains $I(T_\lambda)$.
\end{theorem}

The definition of the quasiregular generalized tangent mapping $T$ is given in Section~\ref{sect:T}, where we also describe some of its geometric properties. In Section~\ref{sect:attr}, we prove Theorem~\ref{thm:attr} by studying the behaviour of $T_\lambda$ as a self-map of $\{(x,y,z):z>0\}$. In Section~\ref{sect:expanding}, we explore the expanding behaviour of the restriction of $T_\lambda$ to the $(x,y)$-plane. A method of defining itineraries on the escaping set $I(T_\lambda)$ is introduced in Section~\ref{sect:itineraries} and this allows us to establish Theorem~\ref{thm:O-(infty)} and Theorem~\ref{thm:sqrt2}. We consider the connectedness of $\overline{O^{-}(\infty)}$ in Section~\ref{sect:conn}, proving Theorem~\ref{thm:connlocus}. In Section~\ref{sect:short-proof}, we quickly deduce Theorem~\ref{thm:expanding} from Theorem~\ref{thm:O-(infty)}. Section~\ref{sect:periodic} contains the proof of Theorem~\ref{thm:periodic}, which again makes use of the itineraries on the escaping set.


\section{The generalized tangent mapping $T$}\label{sect:T}

We will fix the following notation throughout: elements of $\R^3$ will be denoted by $\x=(x,y,z)$; the Euclidean norm will be denoted by $\| \x \|$, with $|\cdot|$ reserved for the modulus of real or complex numbers; and we will write $B(\x,r)$ for an open ball centred at $\x \in \R^3$ of Euclidean radius $r>0$.

\subsection{Construction of $T$}\label{sect:Constr T}

The quasiregular Zorich mapping \cite{Z} (see also \cite[\S I.3.3]{R} and \cite[\S 6.5.4]{IM}) serves as a higher-dimensional analogue of the complex exponential function, and can be defined as follows. First, choose a bi-Lipschitz map 
\[ h:\left[-\frac\pi2, \frac\pi2\right]^2\to\left\{(x,y,z):x^2+y^2+z^2=1, z\ge0\right\}. \]
For our purposes, we shall take 
\[ h(x,y) = \left( \frac{x\sin M(x,y)}{\sqrt{x^2+y^2}}, \frac{y\sin M(x,y)}{\sqrt{x^2+y^2}},\cos M(x,y)\right), \]
where $M(x,y)=\max\{|x|,|y|\}$. Next define $Z:[-\pi/2,\pi/2]^2\times\R\to\{(x,y,z):z\ge0\}$ by $Z(x,y,z)=e^zh(x,y)$. This may then be extended to a mapping $Z:\R^3\to\R^3\setminus\{\0\}$ by repeatedly reflecting in the $x=(k+\frac12)\pi$ and $y=(l+\frac12)\pi$ planes ($k,l\in\Z$) in the domain and in the $z=0$ plane in the image. The resulting Zorich mapping is quasiregular on $\R^3$ and doubly-periodic with periods $(2\pi,0,0)$ and $(0,2\pi,0)$.

We observe that the complex function $\tan z=\frac{i(1-e^{2iz})}{1+e^{2iz}}$ is the composition of a M\"{o}bius map and an exponential function. Define a sense-preserving M\"{o}bius map $A:\R^3\to\R^3\cup\{\infty\}$ by
\[ A(x,y,z) = (2rx, 2ry, 1-2r(z+1)), \]
where
\[ r = r(x,y,z) = \frac{1}{x^2+y^2+(z+1)^2}. \]
We then define our three-dimensional analogue of tangent by
\eqn T(\x) = (A\circ Z)(2\x). \label{defn T} \eqnend
The map $T$ is quasiregular since $A$ and $Z$ are quasiregular. A similar approach could be used to define a variant of $T$ in higher dimensions, but we shall not consider this. Using the expression for the Zorich map discussed above, in the beam $X=[-\pi/4,\pi/4]^2\times\R$ we can write
\begin{multline}
T(x,y,z)= \\
\left(\frac{x\cos M(x,y)\sin M(x,y)}{\sqrt{x^2+y^2}(\cos^2 M(x,y) + \sinh^2 z)},
\frac{y\cos M(x,y)\sin M(x,y)}{\sqrt{x^2+y^2}(\cos^2 M(x,y) + \sinh ^2 z)},
\frac{\sinh z \cosh z}{\cos^2 M(x,y) +  \sinh^2 z} \right). 
\label{teq} \end{multline}
Note that $A$ maps the $(x,y)$-plane to the unit sphere. Hence, we may equivalently define $T$ on $X$ by \eqref{teq}, and then extend $T$ to a mapping $\R^3\to\R^3\cup\{\infty\}$ by reflecting in the $x=\frac\pi2(k+\frac12)$ and $y=\frac\pi2(l+\frac12)$ planes ($k,l\in\Z$) in the domain and inverting in the unit sphere in the image. The inversion referred to here is simply the mapping $\x\mapsto\x/\|\x\|^2$.

\subsection{Geometric properties of $T$}

In this section we describe some of the properties of the mapping $T$ and compare them to those of the tangent function on the complex plane. For example, tangent is $\pi$-periodic while $T$ has periods $(\pi,0,0)$ and $(0,\pi,0)$ by \eqref{defn T}. Furthermore, we see from \eqref{teq} that both functions have fixed points at the origin. 

By construction, the zeroes of $T$ lie at the points $((n+m)\pi/2,(n-m)\pi/2,0)$, where $m,n\in\Z$, and $T$ has poles at $((n+m)\pi/2,(n-m+1)\pi/2,0)$.

In fact, there are two copies of the tangent function embedded in $T$; that is, the restriction of $T$ to either the $(x,z)$-plane or the $(y,z)$-plane gives the standard tangent function on these planes. This can be seen by comparing \eqref{teq} with the trigonometric identity
\eqn \tan(a+ib) = \frac{\cos a\sin a}{(\cos^2a + \sinh^2b)} + i\left(\frac{\sinh b\cosh b}{\cos^2a + \sinh^2b}\right).\label{eq:tan(a+ib)}\eqnend

It is well known that the values $i$ and $-i$ are omitted asymptotic values of the tangent function. It is not hard to see that, for $x,y$ fixed,
\eqn \lim_{z \to \pm\infty} T(x,y,z) = (0,0,\pm1) \label{eq:as value} \eqnend
and that $T$ omits the values $(0,0,\pm1)$. This follows from \eqref{defn T} by noting that the Zorich map has omitted asymptotic values $\0$ and $\infty$, while the M\"{o}bius map $A$ is a homeomorphism of $\R^3\cup\{\infty\}$ with $A(\0)=(0,0,-1)$ and $A(\infty)=(0,0,1)$.

We remark that, in some respects, under the map $T$ the $(x,y)$-plane plays a role similar to that of the real axis under the tangent map, while the $z$-axis plays the role of the imaginary axis. For example, tangent maps the real axis onto $\R\cup\{\infty\}$ and the imaginary axis onto the line segment joining the omitted values $i$ and $-i$. Correspondingly, the image of the $(x,y)$-plane under $T$ is the $(x,y)$-plane plus the point at infinity, and the $z$-axis is mapped by $T$ onto the line segment joining $(0,0,-1)$ to $(0,0,1)$. Moreover, the upper and lower half-planes are completely invariant under the tangent map, as are the half-spaces $\{(x,y,z):z>0\}$ and $\{(x,y,z):z<0\}$ under the mapping $T$. Here we say that a set $U$ is \emph{completely invariant} under a mapping $f$ if $f(U)\subseteq U$ and $f^{-1}(U)\subseteq U$.

We observe that tangent is both an odd function and a real function, so that $\tan(-z)=-\tan z$ and $\tan\overline{z}=\overline{\tan z}$. The related result for the map $T$ is that if $R$ is a reflection in any one of the three co-ordinate planes of $\R^3$, then $T(R(\x))=R(T(\x))$.

Finally, we mention that $T$ is a homeomorphism from $\mathop{\rm int}(X)=(-\pi/4,\pi/4)^2\times\R$ onto the unit ball in $\R^3$.

\section{Dynamics of $T_\lambda$ on the upper and lower half-spaces}\label{sect:attr}

The aim of this section is to prove Theorem \ref{thm:attr}. In the absence of suitable quasiregular versions of the Schwarz Lemma or the Denjoy-Wolff Theorem for this situation, we shall work directly from the expressions \eqref{defn T} and \eqref{teq} for $T$. We begin by noting that the required fixed points exist because $T_\lambda(0,0,z)=(0,0,\lambda\tanh z)$ by \eqref{teq}.

It is useful to observe that if ${\bf x}_0$ is a fixed point of $T_\lambda$, then the basin of attraction $\mathcal{A}({\bf x}_0)$ is completely invariant under $T_\lambda$. Moreover, if ${\bf x}_0$ is an attracting fixed point, then it follows from the definition that $\mathcal{A}({\bf x}_0)$ is an open set on which the iterates of $T_\lambda$ converge to ${\bf x}_0$ locally uniformly. We recall that $\xi_0>0$ is defined by \eqref{defn xi0} and we shall abbreviate $\mathcal{A}((0,0,\pm\xi_0))$ to $\mathcal{A}(\pm\xi_0)$.

The following three propositions will establish Theorem \ref{thm:attr} because of the reflection symmetry of $T$.

\begin{proposition}\label{prop:attr}
If $\lambda >1$, then $T_{\lambda}$ has an attracting fixed point at $(0,0,\xi_0)$. If $0<\lambda <1$, then $T_{\lambda}$ has an attracting fixed point at $\0$.
\end{proposition}

\begin{proposition}\label{A=UHS}
If $\lambda>1$, then $\mathcal{A}(\xi_0) = \{(x,y,z):z>0\}$.
\end{proposition}

\begin{proposition}\label{prop:A(0)}
If $0<\lambda\le1$, then $\{(x,y,z):z>0\}\subseteq\mathcal{A}(\0)$.
\end{proposition}

\begin{proof}[Proof of Proposition \ref{prop:attr}]
Let us fix $\lambda >1$.
The meromorphic function $\tau_\lambda(\zeta)=\lambda\tan\zeta$ has an attracting fixed point at $i\xi_0$ and so there exist $c<1$ and $\varepsilon\in(0,\pi/4)$ such that, for $\zeta\in\mathbb{C}$,
\begin{equation}
|\zeta-i\xi_0|<\varepsilon \quad \Rightarrow \quad |\lambda\tan\zeta - i\xi_0|<c|\zeta-i\xi_0|.  \label{eq:tan zeta}
\end{equation}
Let $(x,y,z)\in B((0,0,\xi_0),\varepsilon)$ and write $M=M(x,y)=\max\{|x|,|y|\}$. We have that
\begin{align}
|(M+iz)-i\xi_0|^2 = M^2 + (z-\xi_0)^2 &\le x^2 + y^2 + (z-\xi_0)^2 \notag \\
&= \|(x,y,z)-(0,0,\xi_0)\|^2<\varepsilon^2. \label{eq:M+iy-iY0}
\end{align}
Since $(x,y,z)\in X$, we see from \eqref{teq} and \eqref{eq:tan(a+ib)} that
\[ T(x,y,z) = \left(\frac{x}{\sqrt{x^2+y^2}}\Real(\tan(M+iz)), \frac{y}{\sqrt{x^2+y^2}}\Real(\tan(M+iz)), \Imag(\tan(M+iz))\right) \]
and hence
\begin{align}
\|T_\lambda(x,y,z)-(0,0,\xi_0)\|^2 &= [\lambda\Real(\tan(M+iz))]^2 + [\lambda\Imag(\tan(M+iz))-\xi_0]^2 \notag \\
&= |\lambda\tan(M+iz)-i\xi_0|^2. \label{eq:T-Y0=tan-Y0}
\end{align}
From \eqref{eq:tan zeta}, \eqref{eq:M+iy-iY0} and \eqref{eq:T-Y0=tan-Y0} we deduce that
\[ \|T_\lambda(x,y,z)-(0,0,\xi_0)\| < c|(M+iz)-i\xi_0| \le c\|(x,y,z)-(0,0,\xi_0)\|. \]
Therefore $(0,0,\xi_0)$ is an attracting fixed point of $T_\lambda$. To prove that $\0$ is an attracting fixed point when $0<\lambda <1$, simply replace $\xi_0$ by 0 in the above argument.
\end{proof}

The following lemma is the key to the proof of Proposition \ref{A=UHS}. For $z>0$, we define
\begin{equation}
 \rho(x,y,z) = \frac{M(x,y)}{z}.   \label{defn rho}
\end{equation}

\begin{lemma}\label{lem:rho}
If $\lambda>0$, $z>0$ and $M(x,y)\ne0$, then we have that 
\[ \rho(T_\lambda(x,y,z)) < \rho(x,y,z). \]
\end{lemma}

\begin{proof}
Recall the notation $X=[-\pi/4,\pi/4]^2\times\R$. We deal first with the case when $(x,y,z)\in X$ with $z>0$ and $M(x,y)\ne0$. Observe that for $s,t>0$,
\[ \frac{\sin s}{s} < 1 < \frac{\sinh t}{t}. \]
Therefore, using \eqref{teq},
\begin{align*}
 \rho(T_\lambda(x,y,z)) &\le \frac{\cos M(x,y) \sin M(x,y)}{\sinh z \cosh z} \\
 &= \frac{\sin 2M(x,y)}{\sinh 2z} < \frac{M(x,y)}{z} = \rho(x,y,z). 
\end{align*} 

Now suppose that $(x,y,z)\in\mathbb{R}^3\setminus X$ with $z>0$, and let $(\hat{x},\hat{y},z)$ be the unique point in $X$ obtained by repeatedly reflecting $(x,y,z)$ in the planes $\{x=\frac\pi2(k+\frac12)\}$ and $\{y=\frac\pi2(l+\frac12)\}$ for $k,l\in\Z$. Since $M(\hat{x},\hat{y})\le \pi/4 < M(x,y)$ we find that
\[ 0\le\rho(\hat{x},\hat{y},z) < \rho(x,y,z). \]
By the construction of $T$, we note that if an even number of reflections are required to move from $(x,y,z)$ to $(\hat{x},\hat{y},z)$, then $T_\lambda(x,y,z)=T_\lambda(\hat{x},\hat{y},z)$, while if this number is odd, then
\[ T_\lambda(x,y,z) = \frac{T_\lambda(\hat{x},\hat{y},z)}{\|T(\hat{x},\hat{y},z)\|^2}. \]
We conclude that in either case
\[ \rho(T_\lambda(x,y,z)) = \rho(T_\lambda(\hat{x},\hat{y},z)). \]
If $M(\hat{x},\hat{y})=0$ then $\rho(T_\lambda(\hat{x},\hat{y},z))=0$ and the result follows. Otherwise, $M(\hat{x},\hat{y})\ne0$ and the proof is completed by applying the first part of the argument to $(\hat{x},\hat{y},z)\in X$.
\end{proof}

\begin{proof}[Proof of Proposition \ref{A=UHS}]
We take $\lambda>1$. Since $\mathcal{A}(\xi_0)$ is open and $T_\lambda(0,0,z)=(0,0,\lambda\tanh z)$, it is clear that
\begin{equation}
\mathcal{A}(\xi_0) \mbox{ contains a neighbourhood of } \{(0,0,z):z>0\}. \label{z-axis in A}
\end{equation}
Thus $B((0,0,\lambda),\delta)\subseteq\mathcal{A}(\xi_0)$ for some $\delta>0$. It then follows from \eqref{eq:as value} that there exists $R>0$ such that $\{(x,y,z):z>R\}\subseteq\mathcal{A}(\xi_0)$.

Write $(T_\lambda)_m$ for the $m$th component function of $T_\lambda$. We now claim that, for $z>0$,
\eqn (T_\lambda)_3(x,y,z)\ge\lambda\tanh z\ge\min\{z,\xi_0\}. \label{T_3>lambda tanh}\eqnend
The second inequality here is clear from the graph of $\lambda\tanh z$ and the definition \eqref{defn xi0}. To prove the first part of \eqref{T_3>lambda tanh}, we again let $(\hat{x},\hat{y},z)$ denote the point in $X$ obtained by reflecting $(x,y,z)$ in the planes $\{x=\frac\pi2(k+\frac12)\}$ and $\{y=\frac\pi2(l+\frac12)\}$ for $k,l\in\Z$. If an even number of reflections are needed, then
\begin{align*}
(T_\lambda)_3(x,y,z)&=(T_\lambda)_3(\hat{x},\hat{y},z) \\
&= \frac{\lambda\sinh z \cosh z}{\cos^2 M(\hat{x},\hat{y}) + \sinh^2 z} \ge\frac{\lambda\sinh z \cosh z}{1 + \sinh^2 z} = \lambda\tanh z.
\end{align*}
Otherwise, if an odd number of reflections are used, then we have that
\begin{align*}
(T_\lambda)_3(x,y,z)&=\frac{(T_\lambda)_3(\hat{x},\hat{y},z)}{\|T(\hat{x},\hat{y},z)\|^2} \\
&= \frac{\lambda\sinh z \cosh z}{\cos^2 M(\hat{x},\hat{y}) + \sinh^2 z} \cdot \frac{\cos^2 M(\hat{x},\hat{y}) + \sinh^2 z}{\sin^2 M(\hat{x},\hat{y}) + \sinh^2 z} \ge \lambda\tanh z,
\end{align*}
so that \eqref{T_3>lambda tanh} holds in either case.

Take $\x_0=(x_0,y_0,z_0)$ with $z_0>0$. We aim to prove that $\x_0\in\mathcal{A}(\xi_0)$. Write $\x_n=(x_n,y_n,z_n)=(T_\lambda)^n(\x_0)$. If, for some $n$, we have that $M(x_n,y_n)=0$ or $z_n>R$, then $\x_n\in\mathcal{A}(\xi_0)$ and we are done. So we may assume that $M(x_n,y_n)\ne0$ and 
\begin{equation}
\min\{z_0,\xi_0\}\le z_n \le R \label{z_n bdd}
\end{equation}
for all $n$, using \eqref{T_3>lambda tanh}. 

By Lemma~\ref{lem:rho}, the sequence $\rho(\x_n)$ is decreasing and so tends to a limit $\rho^*$. We show next that $\rho^*=0$. Since $z_n\le R$ and $M(x_n,y_n)=z_n\rho(\x_n)\le R\rho(\x_0)$, there exists a convergent subsequence $\x_{n_j}\to\x^*$. From \eqref{z_n bdd}, we see that the limit $\x^*$ lies in the upper half-space $\{(x,y,z):z>0\}$, and it follows that both $T_\lambda$ and $\rho$ are continuous near $\x^*$. Therefore
\[ \rho(\x^*) = \lim_{j\to\infty}\rho(\x_{n_j}) = \rho^* \]
and also
\[ \rho(T_\lambda(\x^*)) = \lim_{j\to\infty} \rho(T_\lambda(\x_{n_j})) = \lim_{j\to\infty}\rho(\x_{n_j+1}) = \rho^*. \]
These last two lines contradict Lemma \ref{lem:rho} unless $\x^*$ lies on the $z$-axis, in which case $\rho^*=\rho(\x^*)=0$ as claimed.
We have therefore shown that
\[ \lim_{n\to\infty} M(x_n,y_n) \le \lim_{n\to\infty} R\rho(\x_n) = 0. \]
Recalling \eqref{z-axis in A} and \eqref{z_n bdd}, we now see that $\x_n\in\mathcal{A}(\xi_0)$ for all large $n$. It follows that $\x_0\in\mathcal{A}(\xi_0)$ and therefore $\{(x,y,z):z>0\}\subseteq\mathcal{A}(\xi_0)$. The reverse inclusion is evident from the fact that $\{(x,y,z):z>0\}$ is completely invariant under $T_\lambda$.
\end{proof}

The next lemma will help us to handle the $\lambda=1$ case of Proposition~\ref{prop:A(0)}.

\begin{lemma}\label{V in A(0)}
If $\lambda\le 1$, then there exists $\varepsilon>0$ such that
\[ V=\{(x,y,z):M(x,y)<z/2<\varepsilon\} \subseteq \mathcal{A}(\0). \]
\end{lemma}

\begin{proof}
When $\lambda<1$ the result follows from Proposition \ref{prop:attr} because $\mathcal{A}(\0)$ contains a neighbourhood of $\0$.

Henceforth suppose that $\lambda=1$. For small $z$ and $0<M<z/2$, observe that
\begin{align*}
\cos^2M + \sinh^2z &= (1-M^2 + O(M^4)) + (z^2+O(z^4)) \\
&\ge 1+\frac{3z^2}{4} + O(z^4).
\end{align*}
Therefore, if $z$ is small and $M(x,y)<z/2$, then by \eqref{teq},
\begin{align*}
(T)_3(x,y,z) &= \frac{\sinh z \cosh z}{\cos^2 M(x,y) + \sinh^2 z} \\
&\le \left(z+\frac{z^3}{6} + O(z^5)\right)\left(1+\frac{z^2}{2}+O(z^4)\right)\left(1-\frac{3z^2}{4}+O(z^4)\right) \\
&= z - \frac{z^3}{12} + O(z^5).
\end{align*}
Hence we may choose $\varepsilon>0$ sufficiently small that, for all $(x,y,z)\in V$,
\begin{equation}
(T)_3(x,y,z) \le z - \frac{z^3}{24}. \label{z-z^3/24}
\end{equation}
With this choice of $\varepsilon$, Lemma \ref{lem:rho}, \eqref{defn rho} and \eqref{z-z^3/24} show that $T(V)\subseteq V$. Moreover, \eqref{z-z^3/24} implies that $(T^n)_3(x,y,z)\to 0$ as $n\to\infty$, for all $(x,y,z)\in V$. Therefore $V\subseteq\mathcal{A}(\0)$.
\end{proof}

\begin{proof}[Proof of Proposition \ref{prop:A(0)}]
We let $0<\lambda\le1$. Using Lemma \ref{V in A(0)} together with the fact that $\mathcal{A}(\0)$ is completely invariant and $T_\lambda(0,0,z)=(0,0,\lambda\tanh z)$, we see that
\begin{equation}
\mathcal{A}(\0) \mbox{ contains a neighbourhood of } \{(0,0,z):z>0\}. \label{z-axis in A(0)}
\end{equation}
We take $\x_0=(x_0,y_0,z_0)$ with $z_0>0$ and aim to prove that $\x_0\in\mathcal{A}(\0)$. Write $\x_n=(x_n,y_n,z_n)=(T_\lambda)^n(\x_0)$. Arguing as in the proof of Proposition \ref{A=UHS}, we may assume that $M(x_n,y_n)\ne0$ and $z_n\le R$ for all $n$ and some $R>0$. As before, we aim to show that the decreasing sequence $\rho(\x_n)$ tends to zero. If $\x_n\to \0$ as $n\to\infty$, then $\x_0\in\mathcal{A}(\0)$ and the proposition is proved. Otherwise, since $z_n\le R$ and $M(x_n,y_n)=z_n\rho(\x_n)\le R\rho(\x_0)$ by Lemma~\ref{lem:rho}, there must exist a convergent subsequence $\x_{n_j}\to \x^*$ with limit $\x^*$ in the upper half-space $\{(x,y,z):z>0\}$. In this case, we deduce that $\rho(\x_n)\to0$ as in the proof of Proposition \ref{A=UHS}.

Using Lemma \ref{V in A(0)}, \eqref{defn rho} and \eqref{z-axis in A(0)}, it now follows that $\x_n\in\mathcal{A}(\0)$ for all large $n$.
\end{proof}

\section{Expanding behaviour on the $(x,y)$-plane}\label{sect:expanding}

We have already seen that the $(x,y)$-plane is completely invariant under the map $T$, and in this section we focus on the restriction of $T$ to this plane. Henceforth, we shall refer to this plane as $\R^2$, but we continue to view it as the subset $\{(x,y,0)\}$ of $\R^3$. For balls in $\R^2$ we use the notation $B^2({\bf a},r) = \{\x\in\R^2:\|\x-{\bf a}\|<r\}$. We may sometimes drop the third co-ordinate for brevity, in which case we identify the points $(x,y)=(x,y,0)$.

Using this identification, we define the map $F:\R^2\to\R^2\cup\{\infty\}$ by
\eqn F(x,y)=T(x,y,0) = \left((T)_1(x,y,0),(T)_2(x,y,0)\right) \label{defn F} \eqnend
and we put $F_\lambda (x,y)=\lambda F(x,y)$. We shall be interested in points at which $F_\lambda$ is locally uniformly expanding. For any point $\x$ at which the two-dimensional derivative $DF(\x)$ exists, we write
\[ l(DF(\x)) = \inf_{\|{\bf h}\|=1} \|DF(\x)({\bf h})\|. \]

We make one more definition before stating our next lemma. Let 
\eqn P = \left\{\left(\frac{(n+m)\pi}{2}, \frac{(n-m+1)\pi}{2}, 0 \right) : m,n\in\Z \right\} \label{defn P} \eqnend
denote the common set of poles of $T_\lambda$ and $F_\lambda$.

\begin{lemma}\label{lem:l(DF)}
Let $\lambda>0$. Then, for almost every $\x\in\R^2$,
\[ l(DF_\lambda(\x))\ge\lambda/\sqrt{2}. \] 
Moreover, there exists $\delta>0$ such that, for any $\p\in P$,
\[ l(DF_\lambda(\x))\ge 2 \]
almost everywhere on $B^2(\p,\delta)$.
\end{lemma}

\begin{proof}
We begin by estimating the derivative $DF(\x)$ on the square $X_0=X\cap\R^2=[-\pi/4,\pi/4]^2$. Let us initially assume that $(x,y)\in X_0$ with $0<y<x<\pi/4$, so that
\[ F(x,y)= \left (  \frac{x\tan x}{\sqrt{x^2+y^2}} , \frac{y \tan x}{\sqrt{x^2+y^2}} \right)\]
by \eqref{teq} and \eqref{defn F}. Then the derivative of $F$ at $(x,y)$ is
\[ DF(x,y) = (x^2+y^2)^{-3/2} \left ( \begin{array}{cc}
y^2 \tan x + x(x^2+y^2)(1+\tan ^2 x) & -xy \tan x \\
-xy \tan x +y(x^2+y^2)(1+\tan ^2 x) & x^2 \tan x \\
\end{array}
\right ).\]
The eigenvalues of this matrix are
\[ \mu_1 = \frac{\tan x}{\sqrt{x^2+y^2}} \quad\mbox{and}\quad \mu_2 = \frac{x (1+\tan^2 x)}{\sqrt{x^2+y^2}}, \]
and since $0<y<x<\pi/4$ we have that $\mu_1\ge1/\sqrt{2}$ and $\mu_2\ge1/\sqrt{2}$. Similar calculations show that wherever $DF$ is defined on $X_0$, its eigenvalues are at least $1/\sqrt{2}$. Hence
\[ l(DF_{\lambda}) \geq \lambda / \sqrt{2} \]
almost everywhere in $X_0$.

Consider next $(x,y)\in\R^2$ and, as in Section~\ref{sect:attr}, let $(\hat{x},\hat{y})$ denote the unique point of $X_0$ obtained by repeatedly reflecting $(x,y)$ in the lines $\{x=\frac\pi2(k+\frac12)\}$ and $\{y=\frac\pi2(l+\frac12)\}$ for $k,l\in\Z$. Note that the mapping $(x,y)\mapsto(\hat{x},\hat{y})$ is locally an isometry almost everywhere in $\R^2$. In the case that $(\hat{x},\hat{y})$ is obtained by an even number of reflections, then by \eqref{defn F} and the construction of $T$, we have that $F(x,y)=F(\hat{x},\hat{y})$ and thus we may estimate $DF(x,y)$ by the above argument. In the remaining case, when $(\hat{x},\hat{y})$ is obtained by an odd number of reflections, we find that $F(x,y)= (H\circ F)(\hat{x},\hat{y})$ where $H(\x)=\x/\|\x\|^2$ is inversion in the unit circle. We write $G=H\circ F$ so that in this case, where defined, 
\[ l(DF_{\lambda}(x,y))=\lambda l(DG(\hat{x},\hat{y})). \]
Hence, to prove the first part of the lemma it will now suffice to estimate $DG$ on $X_0$. To this end, we again initially consider $(x,y)\in X_0$ with $0<y<x<\pi/4$. Then
\[ G(x,y) = \left ( \frac{x \cot x}{\sqrt{x^2+y^2}} ,  \frac{y \cot x}{\sqrt{x^2+y^2}} \right ), \]
and we calculate that
\[ DG(x,y) =
(x^2 +y^2)^{-3/2} \left ( \begin{array}{cc}
y^2 \cot x - x(x^2+y^2)\mathop{\rm cosec}^2 x & -xy \cot x \\
-xy \cot x - y(x^2+y^2)\mathop{\rm cosec}^2 x & x^2 \cot x \\
\end{array} \right ). \]
The eigenvalues of $DG$ are
\eqn \mu_3 = \frac{\cot x}{\sqrt{x^2+y^2}} \quad\mbox{and}\quad \mu_4 = \frac{-x\mathop{\rm cosec}^2 x}{\sqrt{x^2+y^2}}, \label{mu3mu4} \eqnend
from which it follows that $|\mu_3|\ge4/(\pi\sqrt{2})$ and $|\mu_4|\ge2/\sqrt{2}$ because $0<y<x<\pi/4$. Similar calculations yield the same eigenvalue estimates at almost every point $(x,y)\in X_0$. Therefore we certainly have that
\[ l(DG(x,y)) \ge 4/(\pi\sqrt{2}) > 1/\sqrt{2} \]
almost everywhere in $X_0$. This establishes the first part of the lemma.

To prove the second statement in the lemma, first observe that if $\p=(p_x,p_y)$ then $(\widehat{p_x},\widehat{p_y})=(0,0)$ and that $F_\lambda(x,y)=\lambda G(\hat{x},\hat{y})$ for all $(x,y)$ in some neighbourhood of $\p$. Given $\lambda>0$, it follows from calculations similar to those leading to \eqref{mu3mu4} that there exists $\delta>0$ such that $|\mu_3|\ge2/\lambda$ and $|\mu_4|\ge2/\lambda$ on $B^2(\0,\delta)$, where $\mu_3$ and $\mu_4$ are again the eigenvalues of $DG$ wherever this is defined. Therefore, $l(DG(\hat{x},\hat{y}))\ge2/\lambda$ almost everywhere on $B^2(\0,\delta)$ and the result follows.
\end{proof}

For $\p\in P$, let $W(\p)$ be the set of points in the $(x,y)$-plane that are nearer to $\p$ than to any other pole. That is,
\[ W(\p) = \{(x,y): \|(x,y)-\p\|<\|(x,y)-\q\| \mbox{ for all } \q\in P\setminus\{\p\}\}; \]
equivalently, writing $\p=(p_x,p_y)$ and recalling \eqref{defn P},
\[ W(\p) = \{(x,y): |x-p_x|+|y-p_y|<\pi/2\}. \]
Theorem \ref{thm:attr} shows that $I(T_\lambda)\subseteq\R^2$. Using \eqref{defn P}, the periodicity of $T_\lambda$, and the fact that $T_\lambda$ maps $\{(x,\pm x,0):x\in\R\}$ into a bounded part of itself, it follows that
\begin{equation}
I(T_\lambda)\subseteq \bigcup_{\p\in P}W(\p). \label{I in UWp}
\end{equation}

\begin{lemma}\label{lem:Sq}
Let $\lambda>0$. For each $\q\in P$ there exists a branch of the inverse of $F_\lambda$ that takes values in $W(\q)$. More precisely, we can define a continuous function
\[ S_\q: (\R^2\cup\{\infty\})\setminus\left\{(x,\pm x) : |x|\le\frac{\lambda}{\sqrt{2}}\right\} \to W(\q) \]
such that $S_\q\circ F_\lambda$ is the identity on $W(\q)$ and $F_\lambda\circ S_\q$ is the identity on the domain of $S_\q$.

In particular, for any $\p,\q\in P$, the function $S_\q$ is defined on $W(\p)$ and so each point of $W(\p)$ has exactly one pre-image under $F_\lambda$ lying in $W(\q)$. Moreover, given $\varepsilon>0$, there exists $R>0$ such that if $\|\p\|>R$ then $S_\q(W(\p))\subseteq B^2(\q,\varepsilon)$.
\end{lemma}

\begin{proof}
Without loss of generality we may take $\lambda=1$. Let $X_0=[-\pi/4,\pi/4]^2$ as before and let 
\[ D= (\R^2\cup\{\infty\})\setminus\left\{(x,\pm x) : |x|\le\frac{1}{\sqrt{2}}\right\}. \]
We claim that $T$ maps $W(\q)$ bijectively onto $D$ and prove this as follows. To each point of $W(\q)$ corresponds a unique point of 
\[ E = \left\{(x,y)\in X_0:x\ne\pm y\right\} \cup \left(\left(\frac{\pi}{4},\frac{3\pi}{4}\right)\times\left(-\frac{\pi}{4},\frac{\pi}{4}\right)\right) \]
obtained by an even number of reflections in the lines $\{x=\frac\pi2(k+\frac12)\}$ and $\{y=\frac\pi2(l+\frac12)\}$ for $k,l\in\Z$. These corresponding pairs of points have the same image under $T$. Hence the claim is equivalent to $T$ being a one-to-one map from $E$ onto $D$. 

Recall the Zorich mapping $Z$ as defined in Section~\ref{sect:Constr T}. It is not difficult to see that the mapping $\x\mapsto Z(2\x)$ is a bijection from $E$ onto
\[ S(\0,1)\setminus\{(x,y,z):x=\pm y, z\ge0\}, \]
where $S(\0,1)$ is the unit sphere in $\R^3$. The M\"{o}bius map $A$ from Section~\ref{sect:Constr T} maps this last set bijectively onto $D$. Thus the claim is now proved by recalling \eqref{defn T}.

From \eqref{defn F} and the above, it follows that we may define inverse branches $S_\q$ as described in the lemma. In fact, using the calculations from the proof of Lemma~\ref{lem:l(DF)}, it can be shown that $F:W(\q)\to D$ is quasiconformal, which implies that the inverse functions $S_\q$ are also quasiconformal; see for example \cite[Corollary II.6.5]{R}.

To prove the final assertion of the lemma, we simply observe that $F_\lambda$ is bounded on $W(\q)\setminus B^2(\q,\varepsilon)$, with this bound independent of the choice of $\q\in P$.
\end{proof}

The next result applies the previous two lemmas to demonstrate that $F_\lambda$ is uniformly expanding in a neighbourhood of any pole.

\begin{lemma}\label{expanding on B}
Given $\lambda>0$, there exists $\varepsilon\in(0,\pi/4)$ such that, for any $\p\in P$ and ${\bf a},{\bf b}\in B^2(\p,\varepsilon)$,
\[ \|F_\lambda({\bf a})-F_\lambda({\bf b})\|\ge 2\|{\bf a}-{\bf b}\|. \]
\end{lemma}

\begin{proof}
Let $\delta>0$ be as given by Lemma \ref{lem:l(DF)} and let $\p\in P$. Choose $R_1>\lambda/\sqrt{2}$ sufficiently large that 
\eqn \{{\bf y}\in\R^2:\|{\bf y}\|>R_1\}\subseteq F_\lambda(B^2(\p,\delta)) \label{defn R1} \eqnend
and find $\varepsilon\in(0,\pi/4)$ such that $\|F_\lambda(\x)\|>2R_1$ whenever $\|\x-\p\|<\varepsilon$. Observe that both $R_1$ and $\varepsilon$ are independent of the choice of $\p\in P$.

Now take ${\bf a},{\bf b}\in B^2(\p,\varepsilon)$ and note that $\|F_\lambda({\bf a})\|,\|F_\lambda({\bf b})\| >2R_1$. Since $\|{\bf a}-{\bf b}\|<\pi/2$, we may assume that $\|F_\lambda({\bf a})-F_\lambda({\bf b})\|<\pi$. Hence, as $R_1$ is large, the line segment joining $F_\lambda({\bf a})$ to $F_\lambda({\bf b})$ must lie in $\{\|{\bf y}\|>R_1\}$. Therefore by \eqref{defn R1}, if ${\bf y}$ is a point on this line segment then $\x:=S_\p({\bf y})\in B^2(\p,\delta)$ and so
\[ \|DS_\p({\bf y})\| = \frac{1}{l(DF_\lambda(\x))} \le \frac12 \]
by Lemma \ref {lem:l(DF)}. It then follows by integration that
\[ \|{\bf a}-{\bf b}\| = \|S_\p(F_\lambda({\bf a})) - S_\p(F_\lambda({\bf b}))\| \le \frac12\|F_\lambda({\bf a}) - F_\lambda({\bf b})\|.  \qedhere \]
\end{proof}

\begin{lemma}\label{lem:diams}
Let $\lambda>0$ and take $\p\in P$ and $U\subseteq W(\p)$. Then for any component $V$ of $F_\lambda^{-1}(U)$, we have $V\subseteq W(\q)$ for some $\q\in P$ and
\[ \diam V \le \frac{\sqrt{2}}{\lambda}\diam U. \]
\end{lemma}

\begin{proof}
From the periodicity of $F_\lambda$ and the fact that $F_\lambda$ maps $\{(x,\pm x):x\in\R\}$ into itself, it follows that $F_\lambda(\partial W(\q))$ is disjoint from $W(\p)$ for all $\q\in P$. Hence we must have that the component $V\subseteq W(\q)$ for some $\q\in P$. By Lemma~\ref{lem:Sq}, $F_\lambda$ has an inverse branch $S_\q:W(\p)\to W(\q)$. 
For ${\bf y}\in W(\p)$, writing $\x=S_\q({\bf y})$ and using Lemma~\ref{lem:l(DF)} yields
\[ \|DS_\q({\bf y})\| = \frac{1}{l(DF_\lambda(\x))} \le \frac{\sqrt{2}}{\lambda}. \]
Noting that $W(\p)$ is convex, we deduce from this that if ${\bf v}_1,{\bf v}_2\in V$, then
\[ \|{\bf v}_1-{\bf v}_2\| = \|S_\q(F_\lambda({\bf v}_1)) - S_\q(F_\lambda({\bf v}_2))\| \le \frac{\sqrt{2}}{\lambda}\|F_\lambda({\bf v}_1) - F_\lambda({\bf v}_2)\| \le \frac{\sqrt{2}}{\lambda}\diam U. \qedhere \]
\end{proof}

\section{Itineraries on the escaping set}\label{sect:itineraries}

We define
\[ \Pi = \{(\p_n)_{n=0}^\infty:\p_n\in P \mbox{ and } \p_n\to\infty \}, \]
so that $\Pi$ is the set of sequences of poles that tend to infinity. The next lemma can be viewed as defining a notion of \emph{itineraries} on the escaping set.

\begin{lemma}\label{itineraries}
For any $\lambda>0$, we can define a one-to-one mapping $\Phi:I(T_\lambda)\to\Pi$ by
\[ \Phi(\x) = (\p_n)_{n=0}^\infty, \]
where $\p_n\in P$ is chosen so that $T_\lambda^n(\x)\in W(\p_n)$.
\end{lemma}

\begin{proof}
We first show that the mapping $\Phi$ is well-defined. Let $\x\in I(T_\lambda)$ and let $n\ge0$ be an integer. Then by \eqref{I in UWp} and the forward invariance of $I(T_\lambda)$, there exists a unique $\p_n\in P$ such that $T_\lambda^n(\x)\in W(\p_n)$. Since $\|T_\lambda^n(\x)-\p_n\|<\pi/2$ and $\x\in I(T_\lambda)$, it follows that $\p_n\to\infty$ as $n\to\infty$.

It remains to show that, given any sequence $(\p_n)\in\Pi$, there exists a unique $\x\in I(T_\lambda)$ such that $\Phi(\x)=(\p_n)$.

Suppose that $(\p_n)\in\Pi$ and let $\varepsilon\in(0,\pi/4)$ be as given by Lemma~\ref{expanding on B}. Let $R>0$ be given by Lemma~\ref{lem:Sq}, except that we increase $R$ if necessary to ensure that $\overline{W(\p)}$ is contained in the domain of $S_\q$ whenever $\|\p\|>R$.

Since $(\p_n)$ tends to infinity, there exists $N$ such that $\|\p_n\|>R$ for all $n\ge N$. For $j\ge0$, write $\q_j=\p_{N+j}$ and define
\[ Y_j = S_{\q_0}\circ\ldots\circ S_{\q_j}(\overline{W(\q_{j+1})}). \]
The $Y_j$ form a nested sequence of non-empty compact sets. By Lemma~\ref{lem:Sq} we have that
\eqn F_\lambda^j(Y_j)=S_{\q_j}(\overline{W(\q_{j+1})})\subseteq\overline{B^2(\q_j,\varepsilon)}. \label{FjYj in B(q,e)} \eqnend
Using this and the fact that the $Y_j$ are nested, we find that, if $0\le k\le j$ and ${\bf a}\in Y_j$, then
\[ F_\lambda^k({\bf a}) \in F_\lambda^k(Y_k) \subseteq \overline{B^2(\q_k,\varepsilon)}. \]
Hence for any pair ${\bf a},{\bf b}\in Y_j$, Lemma \ref{expanding on B} yields that
\begin{eqnarray*}
2^j\|{\bf a}-{\bf b}\| &\le& \|F_\lambda^j({\bf a}) - F_\lambda^j({\bf b})\| \\
&\le& \diam F_\lambda^j(Y_j) \le 2\varepsilon, 
\end{eqnarray*}
where the final inequality here again uses \eqref{FjYj in B(q,e)}. It follows that
\[ \diam Y_j \le 2^{1-j}\varepsilon \to 0 \quad \mbox{as } j\to\infty. \]
We deduce that the intersection of all the sets $Y_j$ contains exactly one point, which we denote by $\x_N$. By construction, and by considering Lemma~\ref{lem:Sq}, we see that $\x_N$ is the unique point such that $F_\lambda^j(\x_N)\in W(\p_{N+j})$ for all $j\ge0$. Finally, we set
\[ \x = S_{\p_0}\circ\ldots\circ S_{\p_{N-1}}(\x_N) \]
and hence $\x$ is the unique point such that $T_\lambda^n(\x)\in W(\p_n)$ for all $n\ge0$. Therefore, $\x\in I(T_\lambda)$ and $\x$ is the unique point satisfying $\Phi(\x)=(\p_n)$, as required.
\end{proof}

\subsection{Proof of Theorem \ref{thm:O-(infty)}}

It is clear from the definition that the set $\Pi$ is uncountable, and so Lemma~\ref{itineraries} shows that the escaping set $I(T_\lambda)$ must also be uncountable. Since $I(T_\lambda)$ is non-empty it must be unbounded and so, by backward invariance, it follows immediately that $O^{-}(\infty)\subseteq\overline{I(T_\lambda)}$.

To prove that $I(T_\lambda)\subseteq\overline{O^{-}(\infty)}$, we let $\x\in I(T_\lambda)$ and take $(\p_n)=\Phi(\x)$. Let $\q_j$, $Y_j$ and $\x_N=T_\lambda^N(\x)$ be as in the proof of Lemma~\ref{itineraries}. By definition, each set $Y_j$ contains a pre-image under $T_\lambda^{j+1}$ of the pole $\q_{j+1}$. Since $\x_N\in\bigcap Y_j$ and $\diam Y_j\to 0$, it follows that $\x_N\in\overline{O^{-}(\infty)}$ and thus $\x\in\overline{O^{-}(\infty)}$ also. We may therefore conclude that $\overline{I(T_\lambda)}=\overline{O^{-}(\infty)}$.

The sets $I(T_\lambda)$ and $O^{-}(\infty)$ are disjoint by definition, and yet we have shown that $\overline{I(T_\lambda)}=\overline{O^{-}(\infty)}$. Using these two facts, it is not hard to show that neither set can have any isolated points.

Finally, we shall show that $I(T_\lambda)$ is totally disconnected. To this end, suppose that $\x$ and ${\bf y}$ belong to the same connected component of $I(T_\lambda)$. Then, by continuity, $T_\lambda^n(\x)$ and $T_\lambda^n({\bf y})$ lie in the same component of $I(T_\lambda)$ for all $n$. Hence \eqref{I in UWp} implies that $\Phi(\x)=\Phi({\bf y})$, because the pairwise disjoint sets $W(\p)$ are relatively open in $\R^2$. Therefore $\x={\bf y}$ by Lemma~\ref{itineraries}.

\subsection{Proof of Theorem \ref{thm:sqrt2}}\label{sect:6-2}

In view of Theorem~\ref{thm:O-(infty)}, our first aim is to show that $\overline{I(T_\lambda)}=\overline{O^{-}(\infty)}=\R^2$ whenever $\lambda>\sqrt{2}$. Suppose that this does not hold, so that there exist $\x\in\R^2$ and $\varepsilon>0$ such that $B(\x,\varepsilon)$ does not intersect $O^{-}(\infty)$. Write $U=B(\x,\varepsilon/2)$ and note that $T_\lambda^n$ is defined, and hence quasiregular, on $U$ for all $n$.
All quasiregular functions are open maps \cite[Theorem I.4.1]{R}, and hence the set $T_\lambda^n(U)$ is open for all $n$. Thus $T_\lambda^n(U)$ meets $W(\p_n)$ for some $\p_n\in P$, because the sets $W(\p)$ form a dense subset of the completely invariant plane $\R^2$. In other words, $U$ intersects some component $V_n$ of $F_\lambda^{-n}(W(\p_n))$. 

Lemma~\ref{lem:diams} shows that 
\[ \diam V_n \le \left(\frac{\sqrt{2}}{\lambda}\right)^n \diam W(\p_n) = \left(\frac{\sqrt{2}}{\lambda}\right)^n \pi,  \]
and so $\diam V_n\to0$ as $n\to\infty$, since $\lambda>\sqrt{2}$. Therefore, if $n$ is sufficiently large, then $V_n\subseteq B(\x,\varepsilon)$. This contradicts the assumption that $B(\x,\varepsilon)$ is disjoint from $O^{-}(\infty)$, because $V_n$ contains a point of $T_\lambda^{-n}(\p_n)$. Hence we have established the fact that $\overline{O^{-}(\infty)}=\R^2$ whenever $\lambda>\sqrt{2}$.

It remains to show that the constant $\sqrt{2}$ is sharp, in the sense that $\overline{I(T_\lambda)}\ne\R^2$ when $\lambda<\sqrt{2}$. We will actually prove somewhat more than this in Lemma~\ref{lem:petals} below. We show that if $\lambda<\sqrt{2}$, then the basin of attraction of the origin contains non-empty open regions when viewed as a subset of the $(x,y)$-plane $\R^2$. The fact that $I(T_\lambda)$ is not dense in $\R^2$ then follows immediately, because $I(T_\lambda)\subseteq\R^2$ by Theorem~\ref{thm:attr}.

We now fix a value $\lambda\in(0,\sqrt{2})$. For $0<\mu<1$, write 
\[ \tau_\mu(x)=\mu\tan x \]
 and let $\phi(\mu)$ denote the smallest positive fixed point of the function $\tau_\mu$. We remark that $\phi:(0,1)\to(0,\pi/2)$ is a continuous decreasing function. Recalling our convention that we identify the points $(x,y,0)=(x,y)$, we define $Q=Q_1\cup Q_2$ where
\[ Q_1 = \left\{(x,\alpha x) : \lambda^2-1<\alpha^2\le 1, \ |x|<\min\left\{\frac{\pi}{4},\phi\left(\frac{\lambda}{\sqrt{1+\alpha^2}}\right)\right\}\right\}, \]
\[ Q_2 = \left\{(\alpha y,y) : \lambda^2-1<\alpha^2\le 1, \ |y|<\min\left\{\frac{\pi}{4},\phi\left(\frac{\lambda}{\sqrt{1+\alpha^2}}\right)\right\}\right\}. \]
The set $Q$ is illustrated for several values of $\lambda$ in Figure~\ref{Fig1}. We may now state the result referred to in the previous paragraph.

\begin{figure}[t]%
\includegraphics[width=\columnwidth]{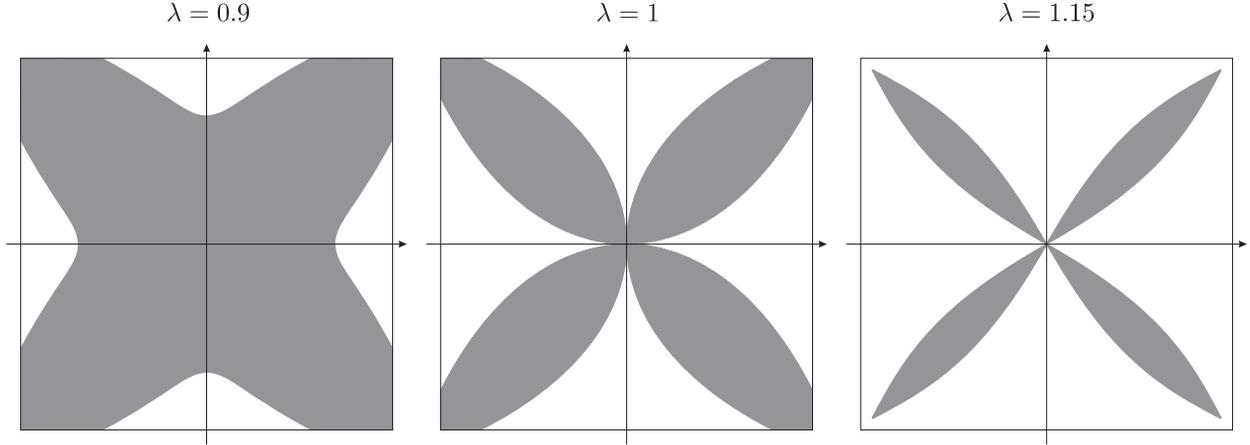}%
\caption{The set $Q$ is shown as the shaded subset of the square $[-\pi/4,\pi/4]^2$.}%
\label{Fig1}%
\end{figure}

\begin{lemma}\label{lem:petals}
For 
$0<\lambda<\sqrt{2}$
and $Q$ as defined above, we have $Q\subseteq\mathcal{A}(\0)$.
\end{lemma}

We briefly delay the proof of Lemma~\ref{lem:petals} to make the following remarks. Although by definition $Q$ contains the origin, the set $Q\setminus\{\0\}$ is non-empty and relatively open as a subset of $\R^2$. When $1\le\lambda<\sqrt{2}$, the origin lies on the boundary of $Q$. If instead $0<\lambda<1$, then $Q$ contains a relative neighbourhood of the origin (in this case we already know that the basin $\mathcal{A}(\0)$ is open by Theorem~\ref{thm:attr}).

\begin{proof}[Proof of Lemma \ref{lem:petals}]
We will just prove that $Q_1\subseteq\mathcal{A}(\0)$, as the proof for $Q_2$ is almost identical.

Take $(x,\alpha x)\in Q_1$ and let $\mu=\lambda/\sqrt{1+\alpha^2}$. Using \eqref{teq}, we find that
\eqn T_\lambda(x,\alpha x,0) = \left(\frac{\lambda\tan x}{\sqrt{1+\alpha^2}},\frac{\alpha\lambda\tan x}{\sqrt{1+\alpha^2}},0\right) = (\tau_\mu(x),\alpha\tau_\mu(x),0). \label{T(x,ax,0)} \eqnend
Since $(x,\alpha x)\in Q_1$, we have that $|x|<\phi(\mu)$. By considering the graph of $\tau_\mu$ and recalling the definition of $\phi(\mu)$, it follows that $|\tau_\mu(x)|<|x|$. Hence \eqref{T(x,ax,0)} shows that $T_\lambda(x,\alpha x,0)\in Q_1$, and we may thus deduce that $T_\lambda(Q_1)\subseteq Q_1$.

Furthermore, for $(x,\alpha x)\in Q_1$, equation \eqref{T(x,ax,0)} implies that
\[ T_\lambda^n(x,\alpha x,0) = (\tau_\mu^n(x),\alpha\tau_\mu^n(x),0). \]
Using the fact that $|x|<\phi(\mu)$, it is not hard to see that $\tau_\mu^n(x)\to0$ as $n\to\infty$. Therefore we conclude that $(x,\alpha x,0)\in\mathcal{A}(\0)$, as required.
\end{proof}

\begin{figure}[p]%
\centering
\includegraphics[width=90mm]{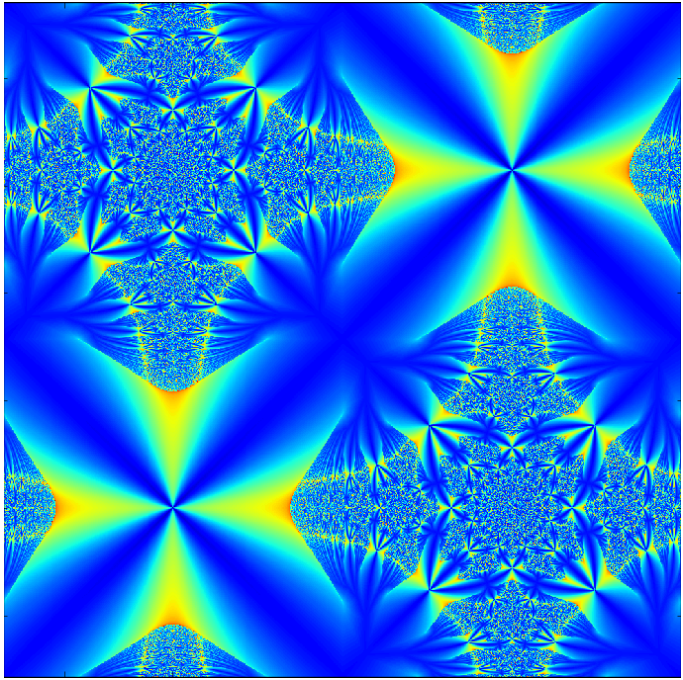}%
\caption{This diagram indicates $\mathcal{A}(\0)$ for $\lambda=0.9$ in $[-\pi/4, 3\pi/4]^2$.}%
\label{Fig2}%
\end{figure}

\begin{figure}[p]%
\centering
\includegraphics[width=90mm]{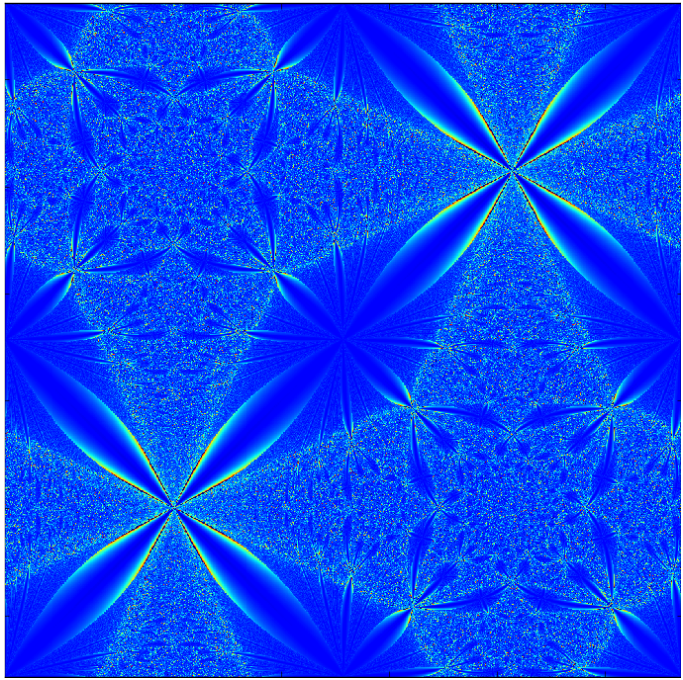}%
\caption{This diagram indicates $\mathcal{A}(\0)$ for $\lambda=1.1107$ in $[-\pi/4, 3\pi/4]^2$. This choice of $\lambda$ is approximately $\pi\sqrt{2}/4$, at which value the tips of the region $Q$ just meet the corners of the square $[-\pi/4,\pi/4]^2$.}
\label{Fig3}%
\end{figure}

\begin{remarks}
\begin{enumerate}
	\item By using \eqref{T(x,ax,0)} and a similar equation for $Q_2$, it can be shown that those points in $(-\pi/4,\pi/4)^2$ that lie on the boundary of $Q$ are fixed points of $T_\lambda$. This means that, when $\lambda\in(\pi/4,\sqrt{2})$, the mapping $T_\lambda$ has a curve of fixed points in the $(x,y)$-plane. (When $\lambda\in(0,\pi/4]$, we find that $Q=(-\pi/4,\pi/4)^2$.)
	\item We mentioned the following questions in Section~\ref{sect:results}: are there values of $\lambda$ for which $\overline{I(T_\lambda)}$ is locally a Cantor set, or for which $\overline{I(T_\lambda)}$ and $\mathcal{A}(\0)$ form a partition of~$\R^3$? When $\lambda\in(\pi/4,\sqrt{2})$, there exists a curve of fixed points of $T_\lambda$. Thus, for these values of $\lambda$, the complement of $\mathcal{A}(\0)$ contains continua. In this case, at least one of the above questions must have a negative answer.	
	\item If $\lambda<1$, then the set $\mathcal{A}(\0)\cap\R^2$ is connected. This may be proved using the fact that $\mathcal{A}(\0)$ is open and contains the lines $\{(x,\pm x, 0):x\in\R\}$.
	\item Figures \ref{Fig2} and \ref{Fig3} show $\mathcal{A}(\0) \cap \R^2$ in the region $[-\pi/4, 3\pi/4]^2$ for $\lambda = 0.9$ and $\lambda = 1.1107$. In these figures, the darker blue regions consist of points which iterate close to $\0$ after few iterations, whereas lighter yellow and red regions indicate that more iterations are needed to get close to $\0$. Since the escaping set is totally disconnected, it is more challenging to compute images which show the escaping set. However, $\mathcal{A}(\0)$ is more amenable to producing visually interesting images; in particular, the region $Q$ is clearly visible in these figures.
Further, Theorem~\ref{thm:connlocus} implies that $\overline{I(T_{\lambda})}$ is not connected in Figure~\ref{Fig2}, but is connected in Figure~\ref{Fig3}.
\end{enumerate}
\end{remarks}

\section{Connectedness of $\overline{O^{-}(\infty)}$}\label{sect:conn}

Recall that Theorem~\ref{thm:O-(infty)} states that $\overline{I(T_\lambda)}=\overline{O^{-}(\infty)}$ for any $\lambda>0$. In this section, we prove Theorem~\ref{thm:connlocus} about the connectedness of this set. In particular, we find that the connectedness locus $\{ \lambda>0 : \overline{I(T_{\lambda})}$ is connected$\}$ is itself connected (see also \cite[Theorem~5.2]{KK}).

The proof of Theorem \ref{thm:connlocus} is contained in the following two results.
If $\lambda \geq 1$, then Lemma~\ref{lem:conn} shows that $\overline{I(T_{\lambda})}=\overline{O^{-}(\infty)}$ is connected. On the other hand, Lemma~\ref{lem:notconn} shows that if $\lambda <1$, then 
$\overline{I(T_{\lambda})}$ contains singleton components. In particular, in this case $\overline{I(T_{\lambda})}$ is not connected.

It is useful to let $L$ be the set of lines $\{(x,\pm x+k\pi,0):x\in\R, k\in\Z\}$. It can be shown that
\eqn T_\lambda(L) = \{(x,\pm x, 0):|x|\le \lambda/\sqrt{2}\} \subseteq L. \label{eqn:L} \eqnend

\begin{lemma}\label{lem:conn}
Let $\lambda \geq 1$. Then $\overline{ O^-(\infty)}$ is connected.
\end{lemma}

\begin{proof}
Let $\lambda\ge1$ and observe that, by Theorem~A, the Julia set of the meromorphic function $\tau_\lambda(z)=\lambda\tan z$ is the real line. In our situation, since $T_{\lambda}$ contains embedded copies of $\tau_\lambda$, this means that the lines $\{x=0\}$ and $\{y=0\}$ in $\R^2$ are contained in $\overline{O^-(\infty)}$. Hence, by periodicity, all the translations of these lines by multiples of $\pi$ in the $x$ and $y$ driections are also contained in $\overline{O^-(\infty)}$. Denote by $C$ the component of $\overline{O^-(\infty)}$ containing all these lines, and note that $C$ is closed and contains all the poles of $T_{\lambda}$. We now aim to show that $O^-(\infty)\subseteq C$, as this will imply that $\overline{O^-(\infty)}$ is equal to $C$ and is therefore connected.

We need to prove that
\eqn T_\lambda^{-n}(\infty)\subseteq C \label{eqn:C} \eqnend
for all $n\in\N$. Note that this holds for $n=1$ because $C$ contains all the poles of $T_{\lambda}$. We proceed by induction, assuming now that \eqref{eqn:C} holds for some particular $n\in\N$.

Let $\x\in T_\lambda^{-(n+1)}(\infty)$ and choose poles $\p_j$ such that $T_\lambda^j(\x)\in W(\p_j)$ for $j=0,1,\ldots,n$; note that $\p_n=T_\lambda^n(\x)$. Such poles $\p_j$ exist because $\partial W(\p)\subseteq L$ for every pole $\p$, and $L$ is disjoint from $O^-(\infty)$ by \eqref{eqn:L}. Choose $\Gamma$ to be a path in $C$ that connects $\p_n$ to $\infty$ and does not intersect $\{(x,\pm x, 0):|x|\le \lambda/\sqrt{2}\}$. Appealing to Lemma~\ref{lem:Sq} and the backward invariance of $\overline{O^-(\infty)}$, we find that
\[ S_{\p_0} \circ \ldots \circ S_{\p_{n-1}} (\Gamma\cup\{\infty\}) \]
is a connected subset of $\overline{O^-(\infty)}$ that contains both $\x$ and the point
\[ S_{\p_0} \circ \ldots \circ S_{\p_{n-1}} (\infty) \in T_\lambda^{-n}(\infty). \]
This latter point lies in $C$ by the induction hypothesis \eqref{eqn:C}. We deduce that $\x\in C$, which completes the induction.
\end{proof}

\begin{lemma}\label{lem:notconn}
Suppose that $0<\lambda<1$. If either ${\bf v}\in O^-(\infty)$ or ${\bf v} \in I(f)$, then $\{ {\bf v}\}$ is a component of $\overline{O^{-}(\infty)}$.
\end{lemma}

\begin{proof}
First suppose that ${\bf v} \in O^-(\infty)$.
The function $T_\lambda$ is quasiregular and all non-constant quasiregular mappings are discrete \cite[Theorem I.4.1]{R}. Hence, it will suffice to prove the result under the assumption that ${\bf v}$ is actually a pole of $T_\lambda$. Let $V$ denote the component of $\overline{O^{-}(\infty)}$ containing ${\bf v}$.

Lemma~\ref{lem:petals} and \eqref{eqn:L}  show that $T_\lambda(L)\subseteq\mathcal{A}(\0)$, by noting that $\lambda/\sqrt{2}<\pi/4<\phi(\lambda/\sqrt{2})$ whenever $0<\lambda<1$. Therefore, $L\subseteq\mathcal{A}(\0)$.

For $n\in\N$, let $\gamma_n\subseteq L$ be a path in the shape of a diamond with vertices at $(\pm n\pi,0,0)$ and $(0,\pm n\pi,0)$. Since $\gamma_n\subseteq L\subseteq \mathcal{A}(\0)$ and $\mathcal{A}(\0)$ is open and disjoint from $O^-(\infty)$, we deduce that $\gamma_n$ does not intersect $\overline{O^{-}(\infty)}$.

Let $D_n$ denote the unbounded component of $\R^2\setminus \gamma_n$. 
The image $T_\lambda(V)$ is contained in $\overline{O^{-}(\infty)}\cup\{\infty\}\subseteq \R^2\cup\{\infty\}$. As $T_\lambda(V)$ is connected and contains $\infty$, but is disjoint from $\gamma_n$, we must have that $T_\lambda(V)\subseteq D_n\cup\{\infty\}$. This holds for all $n$, and hence $T_\lambda(V)=\{\infty\}$. It follows that $V=\{{\bf v}\}$.

Next, suppose instead that ${\bf v} \in I(T_{\lambda})$. 
Let $({\bf p} _n)_{n=0}^{\infty} = \Phi ({\bf v})$ 
be the itinerary of ${\bf v}$.
From the proof of Lemma~\ref{itineraries} recall that, for large $n$,
\[ E_n = S_{{\bf p}_0} \circ \ldots \circ S_{{\bf p}_n} ( \overline{ W({\bf p }_{n+1})})\]
is a nested sequence of compact sets whose intersection consists only of ${\bf v}$. 
Using the first part of this proof,
we have that $\partial W({\bf p}) \subseteq L \subseteq \mathcal{A}(\0)$ for every pole ${\bf p}$.
Therefore, by complete invariance we have $\partial E_n \subseteq \mathcal{A}(\0)$ for all large $n$. 
Since $\mathcal{A}(\0)$ is open and disjoint from $O^-(\infty)$, this means that $\{ {\bf v} \}$ is a component of $\overline{O^-(\infty)}$.
\end{proof}

\section{Proof of Theorem \ref{thm:expanding}}\label{sect:short-proof}

Let $U$ and $E_m$ be as in the statement of the theorem. By Theorem~\ref{thm:O-(infty)} the set $U$ intersects $O^{-}(\infty)$. Hence, for some $m$, the image $T_\lambda^{m-1}(U\setminus E_m)$ contains a set of the form $\{\x\in\R^3:\|\x\|>M\}$ for some $M$. This set contains the beam
\[ \left[k\pi-\frac{\pi}{2},k\pi+\frac{\pi}{2}\right]^2\times\R, \]
for all large $k\in\N$. Therefore, by the periodicity of $T_\lambda$, 
\[ T_\lambda^m(U\setminus E_m) = T_\lambda(\R^3). \]
The proof is completed by recalling the definition \eqref{defn T} and noting that $Z(\R^3)=\R^3\setminus\{\0\}$, while the M\"{o}bius transformation $A$ maps $\R^3\setminus\{\0\}$ onto $\left(\R^3\cup\{\infty\}\right)\setminus\{(0,0,\pm 1)\}$.

\section{Density of periodic points}\label{sect:periodic}

The aim of this section is to provide a proof of Theorem~\ref{thm:periodic}. We fix $\lambda>0$ and, supposing that $\x_0\in I(T_\lambda)$ and $\eta>0$ are given, we seek a periodic point of $T_\lambda$ lying in $B^2(\x_0,\eta)$.

Let $(\p_n)=\Phi(\x_0)$ and write $\x_n=T_\lambda^n(\x_0)$. We will first find a periodic point $\y_0$ in $W(\p_0)$ with iterates that run through the sequence of sets $W(\p_1), W(\p_2), \ldots$  before returning to $\y_0$. By choosing a sufficiently long period for this cycle, we will then be able to show that this periodic point lies near to $\x_0$.

We begin by choosing constants similarly to the proof of Lemma~\ref{itineraries}. In particular, let $\varepsilon>0$ be as in Lemma~\ref{expanding on B} and then take $R$ as given by Lemma~\ref{lem:Sq}, except that we again increase $R$ to ensure that the domain of any inverse branch $S_\q$ contains the closure $\overline{W(\p)}$ for all $\|\p\|>R$. Since $(\p_n)$ tends to infinity, we may choose $N$ such that $\|\p_n\|>R$ whenever $n\ge N$.

We recall the important observation that each inverse branch $S_\q$ is defined on all of $W(\p)$ for every $\p\in P$. Moreover, the image of each $S_\q$ is contained in $W(\q)$ by definition. Using~\eqref{I in UWp}, it follows that the composition $S_{\p_0} \circ \ldots \circ S_{\p_{N-1}}$ is continuous at $\x_N\in I(T_\lambda)$. Hence there exists a large integer $M$ such that
\eqn \left( S_{\p_0}\circ\ldots \circ S_{\p_{N-1}}\right)\left(B^2(\x_N,2^{-M}\pi)\right) \subseteq B^2(\x_0,\eta). \label{eqn:B(x,n)} \eqnend

Next we consider a longer composition of inverse branches. Due to the choices made above, any branch $S_\q$ is defined on $\overline{W(\p_{N+M})}$ because $\|\p_{N+M}\|>R$. Hence we have a continuous function
\[ S_{\p_{N+M}} \circ S_{\p_{0}} \circ S_{\p_{1}} \circ \ldots \circ S_{\p_{N+M-1}} : \overline{W(\p_{N+M})} \to W(\p_{N+M}). \]
An application of the Brouwer Fixed Point Theorem yields a fixed point $\y_{N+M}\in W(\p_{N+M})$ of this composition. 
For $n=0,1,\ldots,N+M-1$, define
\[ \y_n = S_{\p_{n}} \circ \ldots \circ S_{\p_{N+M-1}}(\y_{N+M}), \]
so that $\{\y_0,\ldots,\y_{N+M}\}$ is a periodic cycle for $T_\lambda$ with $\y_n\in W(\p_n)$.

Using the fact that $\|\p_n\|>R$ for all $n\ge N$, the final part of Lemma~\ref{lem:Sq} now shows that in fact
\[ \x_n,\y_n\in B^2(\p_n,\varepsilon) \quad \mbox{for} \quad n\in\{N,\ldots,N+M-1\}, \]
because, for example, $\x_n=S_{\p_n}(\x_{n+1})$. Thus we may apply Lemma~\ref{expanding on B} to obtain
\[ \|\x_{n+1}-\y_{n+1}\| = \|T_\lambda(\x_n) - T_\lambda(\y_n)\| \ge 2\|\x_n-\y_n\|, \]
for $n\in\{N,\ldots,N+M-1\}$, from which we deduce that
\[ \|\x_N-\y_N\| \le 2^{-M}\|\x_{N+M}-\y_{N+M}\| < 2^{-M}\pi. \]
Combining this last line with \eqref{eqn:B(x,n)}, we now see that
\[ \y_0 = S_{\p_0} \circ \ldots \circ S_{\p_{N-1}}(\y_N) \in B^2(\x_0,\eta), \]
which completes the proof of Theorem~\ref{thm:periodic}.

\subsubsection*{Acknowledgment}

The authors wish to thank Dan Goodman for the code used in creating Figures \ref{Fig2} and \ref{Fig3}.

\end{document}